\documentclass[a4paper]{article}

\usepackage[utf8]{inputenc}
\usepackage[english]{babel}
\usepackage{color}
\usepackage{amsmath,amssymb,amsthm}
\usepackage{graphicx,subfigure}
\usepackage[misc,geometry]{ifsym}

\newcommand{\F}{\mathbb{F}}

\newcommand{\R}{\mathbb{R}}
\newcommand{\Z}{\mathbb{Z}}
\newcommand{\PG}{\text{PG}}
\newcommand{\abs}[1]{\left\vert #1\right\vert}

\theoremstyle{definition}

\newtheorem{theorem}{Theorem}
\newtheorem{definition}[theorem]{Definition}
\newtheorem{example}[theorem]{Example}
\newtheorem{lemma}[theorem]{Lemma}

\begin{document}

\title{Power sum polynomials in a discrete tomography perspective}

%\titlerunning{Power sum polynomials in a DT perspective}

\author{Silvia M.C.~Pagani, Silvia Pianta}

\date{\small Università Cattolica del Sacro Cuore\\ via Musei 41, 25121 Brescia, Italy}

%\authorrunning{S.M.C.~Pagani and S.~Pianta}

%\institute{Dipartimento di Matematica e Fisica ``N.~Tartaglia''\\ Universit\`a Cattolica del Sacro Cuore\\ via Musei 41, 25121 Brescia, Italy \\
%\email{\{silvia.pagani,silvia.pianta\}@unicatt.it}}

\maketitle

\begin{abstract}
For a point of the projective space $\PG(n,q)$, its R\'edei factor is the linear polynomial in $n+1$ variables, whose coefficients are the point coordinates. The power sum polynomial of a subset $S$ of $\PG(n,q)$ is the sum of the $(q-1)$-th powers of the R\'edei factors of the points of $S$. The fact that many subsets may share the same power sum polynomial offers a natural connection to discrete tomography. In this paper we deal with the two-dimensional case and show that the notion of ghost, whose employment enables to find all solutions of the tomographic problem, can be rephrased in the finite geometry context, where subsets with null power sum polynomial are called ghosts as well. In the latter case, one can add ghosts still preserving the power sum polynomial by means of the multiset sum (modulo the field characteristic). We prove some general results on ghosts in $\PG(2,q)$ and compute their number in case $q$ is a prime.

\end{abstract}

\section{Introduction}
The aim of the present paper is to investigate the structure and the size of those point sets of $\PG(2,q)$ associated to a given homogeneous polynomial of degree $q-1$, introduced by P.~Sziklai in \cite{sziklai} and named \emph{power sum polynomial}. It is a hard task in general to determine all sets sharing the same power sum polynomial. The ill-posedness of the aforementioned problem enables to construct a link to discrete tomography. There, linear algebra offers a straightforward way of dealing with all images agreeing with the same set of projections (see \cite{hati}). The key role is played by \emph{ghosts}, which constitute the kernel of the linear system in which the problem may be translated.

Motivated by the description of the set of solutions of a tomographic problem, we investigate the recovery of the subsets of $\PG(2,q)$ related to the same power sum polynomial in the following way. We look for an operation on the subsets of $\PG(2,q)$ such that those with associated null polynomial, which we call ghosts as well, constitute the kernel of a suitable function mapping a subset to the corresponding power sum polynomial. A good choice for the operation turns out to be the multiset sum modulo $p$ (the field characteristic). Then, we investigate the algebraic and geometric properties of the set of ghosts and find out the size of the subset of ghosts when $q=p$, namely, for prime fields.

We remark that the ``polynomial technique'' for subsets of a projective space is a well investigated research field in the framework of finite geometries. It consists in the study of the interplay between subsets of a projective space and polynomials over finite fields (see for instance \cite{blokhuis,handbook_b}). It is essentially listed in  three steps: rephrasing the theorem to be proven into a relationship on points, reformulating the (new) theorem in terms of polynomials over finite fields, and perform the calculations. This technique was introduced in the 70’s by L.~R\'edei \cite{redei} and inspired a series of results on blocking sets, directions and codes from the 90’s onward, such as \cite{debeule,gsw,szzs}. The advantage of choosing the power sum polynomial, w.r.t.~the more studied R\'edei polynomial, is the fact that the first one has a lower degree when the subset has more than $q-1$ points.

The paper is organized as follows. Section \ref{sec:def} establishes the framework and gives the main definitions. In Section \ref{sec:tomography} the tomographic problem is recalled, together with a focus on the similarities with the treated problem. Section \ref{sec:generalresults} gives some general results for ghosts in the projective plane, while in Section \ref{sec:p} we set $q=p$ and show the size of the set of ghosts in that case. Section \ref{sec:conclusions} provides possible further work and concludes the paper.

\section{The power sum polynomial}\label{sec:def}
Let $\F_q$ denote the Galois field of order $q=p^h$, $p$ prime, $h\geq1$, whilst $\PG(n,q)$ will refer to the projective space of dimension $n$ over $\F_q$. Let $\mathbf{X}=(X_1,\ldots,X_{n+1})$ be the vector of the variables.

\begin{definition}
Let $P=(p_1,\ldots,p_{n+1})$ be a point of $\PG(n,q)$. The \emph{R\'edei factor} corresponding to $P$ is the linear polynomial $P\cdot\mathbf{X}= p_1X_1+\ldots+p_{n+1}X_{n+1}$.
\end{definition}

The zeros of a R\'edei factor $P\cdot\mathbf{X}$ are the Pl\"ucker coordinates of the hyperplanes through $P$.

The well-known \emph{R\'edei polynomial} of a point set $S$ is defined as the product of the R\'edei factors corresponding to the points of $S$. We are interested in a different polynomial.

\begin{definition}
Let $S=\{P_i\,:\,i=1,\ldots,\abs{S}\}\subseteq \PG(n,q)$ be a point set. The \emph{power sum polynomial} of $S$ is defined as
\begin{displaymath}
G^S(\mathbf{X}):=\sum_{i=1}^{\abs{S}}(P_i\cdot\mathbf{X})^{q-1}.
\end{displaymath}
\end{definition}

The power sum polynomial is therefore a homogeneous polynomial of degree $q-1$. Denote by $\F_q^{q-1}[\mathbf{X}]$ the set of homogeneous polynomials of degree $q-1$ in the variables $\mathbf{X}$ with coefficients in $\F_q$. It is a vector space over $\F_q$ with dimension $\binom{n+q-1}{n}$. For $n=2$, we will write
\begin{eqnarray*}
G^S(X,Y,Z)&=&\sum_{k=1}^{\abs{S}}(a_kX+b_kY+c_kZ)^{q-1}\\
&=&\sum_{k=1}^{\abs{S}}\sum_{i=0}^{q-1}\sum_{j=0}^{q-1-i}\binom{q-1}{i,j}\left(a_kX\right)^{q-1-i-j}\left(b_kY\right)^j\left(c_kZ\right)^i,
\end{eqnarray*}
where ${\displaystyle\binom{q-1}{i,j}=\frac{(q-1)!}{i!j!(q-1-i-j)!}}$ for $i+j\leq q-1$.

The power sum polynomial of the union of disjoint subsets of $\PG(n,q)$ is clearly the sum of the corresponding power sum polynomials. If a hyperplane $(x_1,\ldots,x_{n+1})$ intersects $S$ in $m$ points, then, from the well-known fact that
\begin{displaymath}
\forall \alpha\in\F_q:\alpha^{q-1}=
\left\{
\begin{array}{ll}
0&\text{ if }\alpha=0,\\
1&\text{ otherwise},
\end{array}\right.
\end{displaymath}
it results that $G^S(x_1,\ldots,x_{n+1})=\abs{S}-m$.

As noted in \cite{sziklai}, while there is an one-to-one correspondence between R\'edei polynomials and point sets, the same power sum polynomial may refer to different point sets.

\begin{example}
Consider the Fano plane $\PG(2,2)$ and the subset $S_1=\{(0,0,1)\}$. The corresponding power sum polynomial is $G^{S_1}(X,Y,Z)=Z$, which is shared with other subsets of $\PG(2,2)$, such as $S_2=\{(1,0,1),(1,0,0)\}$ and $S_3=\{(1,0,0),(0,1,0),(1,1,1)\}$.
\end{example}

We address the following problem: Given a homogeneous polynomial $G$ of degree $q-1$, find all sets $S\subseteq\PG(2,q)$ such that $G^S=G$. The choice of considering the two-dimensional case is justified in the next section.

\section{Discrete tomography}\label{sec:tomography}
The problem stated at the end of the previous section belongs to the large class of the inverse problems, where an object has to be retrieved from the measurements of its features, broadly intended. In the same class one can find the \emph{tomographic problem}, whose measurements are the projections of the unknown object along a set of given directions. Tomography originates from the work of J.~Radon \cite{radon}, who proved that, under some regularity hypotheses, a density function defined over $\R^2$ can be reconstructed from its projections taken along angles in the whole range $[0,\pi)$. From \cite{radon} on, the tomographic problem has mainly addressed the two-dimensional case (in diagnostics imaging, for instance, the 3D scan of the body is obtained by stacking 2D slices), as we will do in the present paper. The function to be recovered is also called an \emph{image}.

When a finite set of directions is employed and the image is assumed to be defined on a finite subset of the lattice $\Z^2$, one deals with \emph{discrete tomography}. Directions are lattice directions, i.e., with rational slopes, and projections are obtained by summing up the values of the points which are intersections between lines with given direction and the domain (see Fig.~\ref{fig:proj}). Usually, the codomain of the image is (a subset of) $\Z$.

\begin{figure}[htbp]
\centering
\includegraphics[scale=.8]{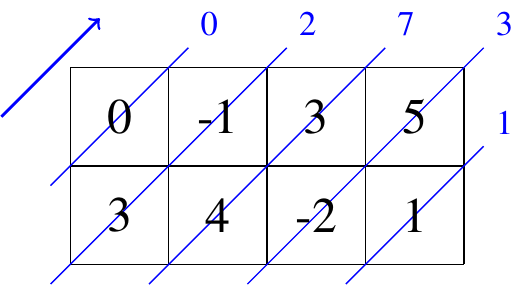}
\caption{A four by two integer-valued image and the projections along a direction.}
\label{fig:proj}
\end{figure}

In general, neither existence nor uniqueness of a solution applies to a discrete tomographic problem, which is ill-posed as most of the inverse problems. In fact, on the one side the available data may be inconsistent due to noise, so that no image exists with prescribed projections. On the other side, when dealing with consistent measurement (i.e., at least one solution exists) many images which agree with the projections could be allowed (see for instance \cite{fishepp}). The lack of uniqueness is caused by the presence of images whose projections along the set of considered directions are zero, and therefore are, in a sense, invisible to measurements. Such images are known as \emph{ghosts} and constitute the kernel of a suitable linear system, obtained by ordering the projections in some way.

Linear algebra and \cite{hati} show how to treat ghosts in order to move among the solutions of a tomographic problem. In fact, every solution of a same problem may be obtained as the sum between a peculiar solution and a suitable ghost (see Fig.~\ref{fig:ghost} for an example). We investigate their counterpart in the power sum polynomial problem, which will be called \emph{ghosts} as well, in particular in the two-dimensional case.

\begin{figure}[htbp]
\centering
\includegraphics[scale=.7]{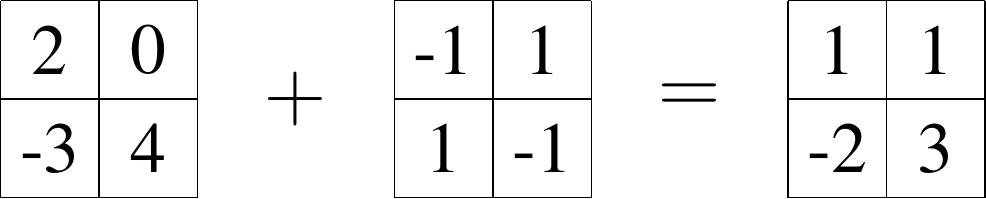}
\caption{Adding a ghost (w.r.t.~the coordinate directions in this case) to an image pro\-duces a new image, whose horizontal and vertical projections equal those of the starting image.}
\label{fig:ghost}
\end{figure}

\subsection{Connections and differences}
Our goal is to determine all point sets of $\PG(2,q)$ sharing the same power sum polynomial. Arguing in analogy with the tomographic framework, the role of the grid (i.e., the domain) and that of an image are played by $\PG(2,q)$ and a point set $S$, respectively. The counterpart of tomographic ghosts is defined as follows.

\begin{definition}
Let $S\subseteq\PG(2,q)$. We say that $S$ is a \emph{ghost} if $G^S(X,Y,Z)\equiv\mathbf{0}$.
\end{definition}

Note that, in the tomographic case, the set of directions has to be specified, while the definition of the power sum polynomial does not depend on other features. We study the structure of ghosts and look for an operation which maps a pair, consisting of a generic point set and a ghost, to another point set with the same power sum polynomial as the previous generic one. The investigation of ghosts therefore completes the search of subsets $S$ such that $G^S=G$ for a given polynomial $G$.

\section{Ghosts and multiset sum in $\PG(2,q)$}\label{sec:generalresults}
As stated in the previous section, we have to study the structure of ghosts. We first prove some general results about classes of ghosts.

In order to prove that a point set $S\subseteq\PG(2,q)$ has the zero polynomial as the corresponding one, we have to find when $G^S$ vanishes identically. We look at the intersections of $S$ with lines, which are the hyperplanes in the plane. In fact, if $\abs{S\cap\ell}=m$ for a line $\ell$, then $G^S(\ell)=\abs{S}-m$, where, by an abuse of notation, we write that the polynomial is computed in $\ell$ instead of the Pl\"ucker coordinates of $\ell$.

Therefore, as $\deg G^S<q$, we seek those point sets whose intersections with all lines of $\PG(2,q)$ have size congruent to that of $S$ modulo the field characteristic $p$. According to \cite[Theorem 13.5]{sziklai}, such point sets $S$ are exactly those intersecting each line in a fixed (mod $p$) number of points. So, our ghosts are the \emph{generalized Vandermonde sets} defined in \cite{sziklai}.

\begin{theorem}\label{teo:complement}
Let $S\subseteq\PG(2,q)$ be a ghost. Then $\PG(2,q)\backslash S$ is a ghost.
\end{theorem}

\proof Since $S$ has constant intersection size with lines, the same holds true for its complement $\PG(2,q)\backslash S$. \endproof

As a consequence, both the empty set and the whole projective plane are ghosts.

\begin{theorem}\label{teo_pencil}
A partial pencil $\mathcal{P}$ of $\lambda p+1$ lines, $\lambda=0,\ldots,p^{h-1}$, is a ghost. Consequently, a set of $q-\lambda p$ lines through a point $P$ minus $P$ is a ghost.
\end{theorem}

\begin{proof}
Every line $\ell$ meets $\mathcal{P}$ in either one, $\lambda p+1$ or $q+1$ points. In all cases $m=1\mod p$ and
\begin{displaymath}
G^{\mathcal{P}}(\ell)=(\lambda p+1)q+1-m=0\mod p.
\end{displaymath}
\end{proof}

Therefore, every line is a ghost, as well as any affine plane contained in $\PG(2,q)$.

It would be desirable to define a binary operation in order to endow the subsets of $\PG(2,q)$ with a group structure, so that the set of ghosts is stable under such an operation. Unfortunately, the usual set-theoretical union is not a good choice. Consider for instance the two lines $\ell_1,\ell_2$ of $\PG(2,2)$, whose points have coordinates satisfying $X=0$ and $Y=0$, respectively. Lines are ghosts of $\PG(2,2)$; their union consists of the five points $(0,1,0)$, $(0,0,1)$, $(0,1,1)$, $(1,0,0)$, $(1,1,0)$ and the corresponding power sum polynomial is
\begin{displaymath}
G^{\ell_1\cup\ell_2}=Y+Z+(Y+Z)+X+(X+Y)=Y,
\end{displaymath}
so the union of two ghosts is not a ghost in general. To deal with a deeper algebraic structure, we need to consider a different kind of operation.

\subsection{Multiset sum}
The structure of all solutions of a same tomographic problem can be described by linear algebra as the sum between a peculiar solution and the vector subspace of ghosts (see \cite{hati}). Concerning power sum polynomials, the usual set-theoretical union is not enough and has to be replaced with the multiset sum.

Let $A$ be a multiset, namely, a set whose elements can appear with multiplicity greater than one. Denote by $m_A(x)$ the multiplicity of the element $x$ in $A$, which is zero if $x\notin A$. The \emph{multiset sum} of two multisets $A$ and $B$ is a multiset $C$ such that, for any element $x$,
\begin{displaymath}
m_C(x)=m_A(x)+m_B(x).
\end{displaymath}
Denote the set of subsets $S$ of $\PG(2,q)$, where each point $P$ is counted $m_S(P)\mod p$ times, by $p^{\PG(2,q)}$. Of course $p^{\PG(2,q)}$ has $p^{q^2+q+1}$ elements. Moreover, consider the binary operation $\uplus_p$, which is the multiset sum modulo $p$.

\begin{lemma}
$\left(p^{\PG(2,q)},\uplus_p\right)$ is an abelian $p$-group.
\end{lemma}

\proof The operation $\uplus_p$ is clearly internal to $p^{\PG(2,q)}$, associative and commutative. The identity element is the empty set and the inverse of a subset $S$, whose element $P_i$ is counted $\lambda_i$ times ($i=1,\ldots,\abs{S}$), is the subset where $P_i$ is counted $p-\lambda_i$ times.

Also, $p^{\PG(2,q)}$ is a $p$-group since every non-identity element has period $p$. \qed

It follows that the map
\begin{displaymath}
\varphi:\left\{
\begin{array}{rcl}
p^{\PG(2,q)} & \longrightarrow & \F_q^{q-1}[X,Y,Z],\\
S & \longmapsto & G^S
\end{array}\right.
\end{displaymath}
preserves the group operation, i.e.,
\begin{displaymath}
\varphi\left(S_1\uplus_p S_2\right)=G^{S_1\uplus_p S_2}=G^{S_1}+G^{S_2}=\varphi(S_1)+\varphi(S_2).
\end{displaymath}
Therefore $\varphi$ is a homomorphism between abelian groups. In particular, if $A\subseteq B$, then $G^{B\setminus A}=G^B-G^A$.

The kernel of $\varphi$ is the set of ghosts $\mathcal{G}$, which is indeed a subgroup of $p^{\PG(2,q)}$. Moreover
\begin{displaymath}
p^{\PG(2,q)}\,/\,\ker\varphi\,\cong\text{ im}\varphi.
\end{displaymath}

We now extend the results of Theorems \ref{teo:complement} and \ref{teo_pencil} to multisets. For a subset $A$ of a multiset $B$, we define the complement of $A$ in $B$ as the multiset where each element $b\in B$ is counted $m_B(b)-m_A(b)$ times.

\begin{theorem}
Let $S$ be a ghost. Then the complement of $S$ is a ghost. Moreover, lines and (multiset) sums of lines are ghosts of $\PG(2,q)$.
\end{theorem}

\proof It results $G^{\PG(2,q)\setminus S}=G^{\PG(2,q)}-G^S=0-0=0$. The statement about lines follows from Theorem \ref{teo_pencil} and the properties of $\uplus_p$. \qed

\section{The case $q=p$}\label{sec:p}
From now on we focus on the case $q=p^h=p$ (i.e., $h=1$), namely, we refer to prime fields. In order to get information about the size of the subgroup of ghosts, we argue on the size of the image of $\varphi$.

\subsection{Surjectivity of $\varphi$}
We now prove that, for $q=p$ prime, the dimension of im$\varphi$ is $\binom{p+1}{2}$, namely, the function $\varphi$ is onto.

If $p=2$, it is trivial to prove that the images of points $(1,0,0)$, $(0,1,0)$, $(0,0,1)$ constitute a basis of $\F_2^1[X,Y,Z]$. For $p$ odd, we will prove that the set consisting of the images of the $\binom{p+1}{2}$ points
\begin{equation}\label{eq:base}
\begin{array}{l}
(1,0,0),(1,0,1),(1,0,2),\ldots,(1,0,p-1),\\
(1,1,0),(1,1,1),\ldots,(1,1,p-2),\\
\ldots,\\
(1,p-2,0),(1,p-2,1),\\
(1,p-1,0)
\end{array}
\end{equation}
is linearly independent. This can be done by showing that the matrix having as rows the components (w.r.t.~the canonical basis of $\F_p^{p-1}[X,Y,Z]$) of the images of the above points is non-singular. Each entry of the matrix is the coefficient of a certain monomial in the development of $(X+bY+cZ)^{p-1}$, where $(1,b,c)$ is one of the above points. Since from the column corresponding to the monomial $X^{p-1-i-j}Y^jZ^i$ we can extract the coefficient $\binom{p-1}{i,j}$ and this does not interfere with the fact that the determinant is either zero or non-zero, we will omit the coefficients and simply write the value of $b^jc^i$. By reordering rows and columns in a suitable way we get the matrix in Table \ref{tab2}, which is a block lower triangular matrix. Its determinant is given by the product of the determinants of its blocks. The four blocks are: the one by one upper leftmost block $(1)$, two equal blocks
\begin{displaymath}
\left(
\begin{array}{ccccc}
1&1&1&\ldots&1\\
2&2^2&2^3&\ldots&1\\
\vdots&&&&\\
p-1&(p-1)^2&(p-1)^3&\ldots&1
\end{array} \right)
\end{displaymath}
and the remaining block of order $\frac{(p-2)(p-1)}{2}$.

\begin{table}[tb]
{\scriptsize
\begin{center}
\begin{tabular}[htb]{c|ccccccccccccc}
&1&$b$&$b^2$&\ldots&$b^{p-1}$&$c$&$c^2$&\ldots&$c^{p-1}$&$bc$&$b^2c$&\ldots&$bc^{p-2}$\\
\hline
$(1,0,0)$&1&0&\ldots&&&&&&&&&&\\
$(1,1,0)$&1&1&1&\ldots&1&0&\ldots&&&&&&\\
$(1,2,0)$&1&2&$2^2$&\ldots&1&0&\ldots&&&&&&\\
\vdots&\vdots&&&&&&&&&&&&\\
$(1,p-1,0)$&1&$p-1$&$(p-1)^2$&\ldots&1&0&\ldots&&&&&&\\
$(1,0,1)$&1&0&0&\ldots&0&1&1&\ldots&1&0&\ldots&&\\
$(1,0,2)$&1&0&0&\ldots&0&2&$2^2$&\ldots&1&0&\ldots&&\\
\vdots&\vdots&&&&&&&&&&&&\\
$(1,0,p-1)$&1&0&0&\ldots&0&$p-1$&$(p-1)^2$&\ldots&1&0&\ldots&&\\
$(1,1,1)$&1&1&1&\ldots&1&1&1&\ldots&1&1&1&\ldots&1\\
\vdots&\vdots&&&&&&&&&&&&\\
$(1,p-2,1)$&1&$p-2$&$(p-2)^2$&\ldots&1&1&1&\ldots&1&$p-2$&$(p-2)^2$&\ldots&$p-2$\\
\end{tabular}
\caption{The initial matrix.}\label{tab2}
\end{center}
}
\end{table}

The first three blocks are non-singular (in particular, the second and the third blocks are Vandermonde matrices), so the matrix has non-zero determinant if and only if the fourth block has.

The columns of the fourth block are indexed as $bc,b^2c,bc^2,\ldots,bc^{p-2}$, so we can extract a factor $bc$ from every row to get, after reordering rows and columns suitably, the matrix in Table \ref{tab3}.

\begin{table}[tb]
{\tiny
\begin{center}
\begin{tabular}[htb]{c|cccccccccccc}
&1&$b$&$b^2$&\ldots&$b^{p-3}$&$c$&$bc$&\ldots&$b^{p-4}c$&$c^2$&\ldots&$c^{p-3}$\\
\hline
$(1,1,1)$&1&1&1&\ldots&&&&&&&&\\
$(1,2,1)$&1&2&$2^2$&\ldots&$2^{p-3}$&1&2&\ldots&$2^{p-4}$&1&\ldots&1\\
$(1,3,1)$&1&3&$3^2$&\ldots&$3^{p-3}$&1&3&\ldots&$3^{p-4}$&1&\ldots&1\\
\vdots&\vdots&&&&&&&&&&&\\
$(1,p-2,1)$&1&$p-2$&$(p-2)^2$&\ldots&$(p-2)^{p-3}$&1&$p-2$&\ldots&$(p-2)^{p-4}$&1&\ldots&1\\
$(1,1,2)$&1&1&1&\ldots&1&2&2&\ldots&2&$2^2$&\ldots&$2^{p-3}$\\
$(1,2,2)$&1&2&$2^2$&\ldots&$2^{p-3}$&2&$2\cdot2$&\ldots&$2^{p-4}\cdot2$&$2^2$&\ldots&$2^{p-3}$\\
\vdots&\vdots&&&&&&&&&&&\\
$(1,p-3,2)$&1&$p-3$&$(p-3)^2$&\ldots&$(p-3)^{p-3}$&2&$(p-3)\cdot2$&\ldots&$(p-3)^{p-4}\cdot2$&$2^2$&\ldots&$2^{p-3}$\\
$(1,1,3)$&1&1&1&\ldots&1&3&3&\ldots&3&$3^2$&\ldots&$3^{p-3}$\\
\vdots&\vdots&&&&&&&&&&&\\
$(1,1,p-2)$&1&1&1&\ldots&1&$p-2$&$p-2$&$p-2$&\ldots&$(p-2)^2$&\ldots&$(p-2)^{p-3}$\\
\end{tabular}
\caption{The $\frac{(p-2)(p-1)}{2}$ lower leftmost block.}\label{tab3}
\end{center} }
\end{table}

In order to prove that the determinant of the matrix in Table \ref{tab3} is non-zero, we apply some basic operations to its rows in order to obtain a block upper triangular matrix, whose blocks are Vandermonde matrices and which is singular if and only if the starting matrix is. The procedure consists of $p-2$ steps; in each of them we consider some rows as \emph{pivotal} and substitute to a non-pivotal row a linear combination of its and the corresponding pivotal row, in order to have zeroes as entries in the columns corresponding to a certain power of $c$.

Denote by $(1,b,c)^{(n)}$ the row corresponding to the image of the point $(1,b,c)$ at the $n$-th step of the procedure. It is defined recursively as follows:
\begin{eqnarray*}
(1,b,c)^{(0)} & : & \text{the row corresponding to }\varphi(1,b,c),\\
(1,b,c)^{(1)} & := & (1,b,c)^{(0)}-(1,b,1)^{(0)}, \quad c\geq2,\\
(1,b,c)^{(n)} & := & \frac{(1,b,c)^{(n-1)}}{c-(n-1)}-(1,b,n)^{(n-1)}, \quad n\geq2,c\geq n+1,
\end{eqnarray*}
where the subtraction is intended to be element-wise. The third entry is at least $n+1$ since at Step $n$ the pivotal rows are $(1,b,n)^{(n-1)}$, $b=1,\ldots,p-1-n$, so only rows such that $c\geq n+1$ will be modified.

After each step, the matrix becomes block upper triangular consisting of two blocks: the upper leftmost block is a Vandermonde matrix and the other block has to be treated at the next step. Rows in the matrix in Table \ref{tab3} are ordered such that the first $p-2$ rows ($(1,b,1)$, $b=1,\ldots,p-2$) are the pivotal ones in the first step, the following $p-3$ ones ($(1,b,2)$, $b=1,\ldots,p-3$) will be the pivotal rows in the second step after the operations, and so on.

At Step $n$, the row $(1,b,c)^{(n)}$ has the entry corresponding to the column $b^{\lambda}c^{\,\mu}$ equal to:
\begin{itemize}
\item for $n=1$: $b^{\lambda}\left(c^{\mu}-1\right)$;
\item for $n=2n'$, $n'\geq1$:
\begin{displaymath}
\begin{array}{l}
\displaystyle
b^{\lambda}\sum_{i_1=0}^{\mu-2}\left(2^{\mu-1-i_1}-1\right)\sum_{i_2=0}^{i_1-2} \left(4^{i_1-1-i_2}-3^{i_1-1-i_2}\right)\sum_{i_3=0}^{i_2-2}\ldots \\
\displaystyle
\cdot\sum_{i_{n'}=0}^{i_{n'-1}-2}(c-n) \left(n^{i_{n'-1}-1-i_{n'}}-(n-1)^{i_{n'-1}-1-i_{n'}}\right)c^{i_{n'}};
\end{array}
\end{displaymath}
\item for $n=2n'+1$, $n'\geq1$:
\begin{displaymath}
\begin{array}{l}
\displaystyle
b^{\lambda}\sum_{i_1=0}^{\mu-2}\left(2^{\mu-1-i_1}-1\right)\sum_{i_2=0}^{i_1-2} \left(4^{i_1-1-i_2}-3^{i_1-1-i_2}\right)\sum_{i_3=0}^{i_2-2}\ldots\\
\displaystyle
\cdot\sum_{i_{n'}=1}^{i_{n'-1}-2} \left((2n')^{i_{n'-1}-1-i_{n'}}-(2n'-1)^{i_{n'-1}-1-i_{n'}}\right)\left(c^{i_{n'}}-n^{i_{n'}}\right).
\end{array}
\end{displaymath}
\end{itemize}
The above formulas may be proven by induction on $n$. Note that, for $\mu<n$, there are empty summations and therefore the entry is zero. Moreover, all elements of row $(1,b,c)^{(n)}$ are divisible by $c-n$ even over the integers.

In particular, the pivotal rows at the next step, say $n+1$, have a Vandermonde submatrix in the columns corresponding to the power $b^{\lambda}c^{n}$, $\lambda=0,\ldots,p-2-n$ (it is obtained by substituting $\mu=n$ and $c=n+1$ in the above formulas; note that, for $\mu=n$, all summations reduce to a single term).

At the final step $p-2$, the resulting matrix is block triangular and each block is a Vandermonde matrix with no repeated rows, so non-singular. We can conclude that the images of points in \eqref{eq:base} are linearly independent and then im$\varphi$ has dimension $\binom{p+1}{2}$.

In the next example we show how the procedure works.

\begin{example}\label{ex7}
Set $p=7$ and consider the $\frac{5\cdot6}{2}=15$ points
\begin{eqnarray*}
&&(1,1,1),(1,1,2),(1,1,3),(1,1,4),(1,1,5),(1,2,1),(1,2,2),(1,2,3),(1,2,4),\\
&&(1,3,1),(1,3,2),(1,3,3),(1,4,1),(1,4,2),(1,5,1),
\end{eqnarray*}
which are those, among the $\binom{7+1}{2}=28$ points whose images constitute a basis for im$\varphi$, having no zero coordinate. This means that the above points are those whose images constitute the block on which the procedure is applied.

The initial matrix is
\begin{displaymath}
\begin{array}{c|ccccccccccccccc}
&1&b&b^2&b^3&b^4&c&bc&b^2c&b^3c&c^2&bc^2&b^2c^2&c^3&bc^3&c^4\\
\hline
(1,1,1)&1&1&1&1&1&1&1&1&1&1&1&1&1&1&1\\
(1,2,1)&1&2&2^2&2^3&2^4&1&2&2^2&2^3&1&2&2^2&1&2&1\\
(1,3,1)&1&3&3^2&3^3&3^4&1&3&3^2&3^3&1&3&3^2&1&3&1\\
(1,4,1)&1&4&4^2&4^3&4^4&1&4&4^2&4^3&1&4&4^2&1&4&1\\
(1,5,1)&1&5&5^2&5^3&5^4&1&5&5^2&5^3&1&5&5^2&1&5&1\\
(1,1,2)&1&1&1&1&1&2&2&2&2&2^2&2^2&2^2&2^3&2^3&2^4\\
(1,2,2)&1&2&2^2&2^3&2^4&2&2\cdot2&2^2\cdot2&2^3\cdot2&2^2&2\cdot2^2&2^2\cdot2^2&2^3&2\cdot2^3&2^4\\
(1,3,2)&1&3&3^2&3^3&3^4&2&3\cdot2&3^2\cdot2&3^3\cdot2&2^2&3\cdot2^2&3^2\cdot2^2&2^3&3\cdot2^3&2^4\\
(1,4,2)&1&4&4^2&4^3&4^4&2&4\cdot2&4^2\cdot2&4^3\cdot2&2^2&4\cdot2^2&4^2\cdot2^2&2^3&4\cdot2^3&2^4\\
(1,1,3)&1&1&1&1&1&3&3&3&3&3^2&3^2&3^2&3^3&3^3&3^4\\
(1,2,3)&1&2&2^2&2^3&2^4&3&2\cdot3&2^2\cdot3&2^3\cdot3&3^2&2\cdot3^2&2^2\cdot3^2&3^3&2\cdot3^3&3^4\\
(1,3,3)&1&3&3^2&3^3&3^4&3&3\cdot3&3^2\cdot3&3^3\cdot3&3^2&3\cdot3^2&3^2\cdot3^2&3^3&3\cdot3^3&3^4\\
(1,1,4)&1&1&1&1&1&4&4&4&4&4^2&4^2&4^2&4^3&4^3&4^4\\
(1,2,4)&1&2&2^2&2^3&2^4&4&2\cdot4&2^2\cdot4&2^3\cdot4&4^2&2\cdot4^2&2^2\cdot4^2&4^3&2\cdot4^3&4^4\\
(1,1,5)&1&1&1&1&1&5&5&5&5&5^2&5^2&5^2&5^3&5^3&5^4
\end{array}
\end{displaymath}

The upper leftmost 5 by 5 block is a Vandermonde matrix and the corresponding rows will be the pivotal ones at Step 1. The other rows will be manipulated so that the first five entries are zero. Therefore, matrix at Step 1 is
\begin{displaymath}
\begin{array}{c|ccccccccccccccc}
&1&b&b^2&b^3&b^4&c&bc&b^2c&b^3c&c^2&bc^2&b^2c^2&c^3&bc^3&c^4\\
\hline
(1,1,1)&1&1&1&1&1&*&\ldots&&&&&&&&\\
(1,2,1)&1&2&2^2&2^3&2^4&*&\ldots&&&&&&&&\\
(1,3,1)&1&3&3^2&3^3&3^4&*&\ldots&&&&&&&&\\
(1,4,1)&1&4&4^2&4^3&4^4&*&\ldots&&&&&&&&\\
(1,5,1)&1&5&5^2&5^3&5^4&*&\ldots&&&&&&&&\\
(1,1,2)^{(1)}&0&0&0&0&0&1&1&1&1&3&3&3&7&7&15\\
(1,2,2)^{(1)}&0&0&0&0&0&1&2&2^2&2^3&3&2\cdot3&2^2\cdot3&7&2\cdot7&15\\
(1,3,2)^{(1)}&0&0&0&0&0&1&3&3^2&3^3&3&3\cdot3&3^2\cdot3&7&3\cdot7&15\\
(1,4,2)^{(1)}&0&0&0&0&0&1&4&4^2&4^3&3&4\cdot3&4^2\cdot3&7&4\cdot7&15\\
(1,1,3)^{(1)}&0&0&0&0&0&2&2&2&2&8&8&8&26&26&80\\
(1,2,3)^{(1)}&0&0&0&0&0&2&2\cdot2&2^2\cdot2&2^3\cdot2&8&2\cdot8&2^2\cdot8&26&2\cdot26&80\\
(1,3,3)^{(1)}&0&0&0&0&0&2&3\cdot2&3^2\cdot2&3^3\cdot2&8&3\cdot8&3^2\cdot8&26&3\cdot26&80\\
(1,1,4)^{(1)}&0&0&0&0&0&3&3&3&3&15&15&15&63&63&255\\
(1,2,4)^{(1)}&0&0&0&0&0&3&2\cdot3&2^2\cdot3&2^3\cdot3&15&2\cdot15&2^2\cdot15&63&2\cdot63&255\\
(1,1,5)^{(1)}&0&0&0&0&0&4&4&4&4&24&24&24&124&124&624
\end{array}
\end{displaymath}
where the asterisk means any value. In Step 2 we focus on the 10 by 10 lower rightmost block. Rows corresponding to $(1,b,2)^{(1)}$, $b=1,2,3,4$, form a Vandermonde matrix in the first four columns and they will be the pivotal ones. The (sub-)matrix at Step 2 is
\begin{displaymath}
\begin{array}{c|cccccccccc}
&c&bc&b^2c&b^3c&c^2&bc^2&b^2c^2&c^3&bc^3&c^4\\
\hline
(1,1,2)^{(1)}&1&1&1&1&*&\ldots&&&&\\
(1,2,2)^{(1)}&1&2&2^2&2^3&*&\ldots&&&&\\
(1,3,2)^{(1)}&1&3&3^2&3^3&*&\ldots&&&&\\
(1,4,2)^{(1)}&1&4&4^2&4^3&*&\ldots&&&&\\
(1,1,3)^{(2)}&0&0&0&0&1&1&1&6&6&25\\
(1,2,3)^{(2)}&0&0&0&0&1&2&2^2&6&2\cdot6&25\\
(1,3,3)^{(2)}&0&0&0&0&1&3&3^2&6&3\cdot6&25\\
(1,1,4)^{(2)}&0&0&0&0&2&2&2&14&14&70\\
(1,2,4)^{(2)}&0&0&0&0&2&2\cdot2&2^2\cdot2&14&2\cdot14&70\\
(1,1,5)^{(2)}&0&0&0&0&3&3&3&24&24&141
\end{array}
\end{displaymath}
At Step 3 we consider the 6 by 6 lower rightmost submatrix, whose pivotal rows will be $(1,b,3)^{(2)}$, $b=1,2,3$, and the remaining rows are modified as follows:
\begin{displaymath}
\begin{array}{c|cccccc}
&c^2&bc^2&b^2c^2&c^3&bc^3&c^4\\
\hline
(1,1,3)^{(2)}&1&1&1&*&\ldots&\\
(1,2,3)^{(2)}&1&2&2^2&*&\ldots&\\
(1,3,3)^{(2)}&1&3&3^2&*&\ldots&\\
(1,1,4)^{(3)}&0&0&0&1&1&10\\
(1,2,4)^{(3)}&0&0&0&1&2&10\\
(1,1,5)^{(3)}&0&0&0&2&2&22
\end{array}
\end{displaymath}
The matrix at Step 4 is the following one:
\begin{displaymath}
\begin{array}{c|ccc}
&c^3&bc^3&c^4\\
\hline
(1,1,4)^{(3)}&1&1&*\\
(1,2,4)^{(3)}&1&2&*\\
(1,1,5)^{(4)}&0&0&1
\end{array}
\end{displaymath}
At Step $5=7-2$ the one by one matrix $(1)$ is left. Therefore, the initial matrix has been transformed into a block triangular matrix, whose blocks are Vandermonde matrices with no rows repeated, and so non-singular. This means that also the initial matrix is non-singular. Combined with the fact that even the other three blocks of the matrix, whose rows are the images of the 28 points as in \eqref{eq:base}, are non-singular, it results that im$\varphi$ has dimension 28.
\end{example}

The knowledge of $\dim$ im$\varphi$ enables us to compute the size of the ghost subgroup $\mathcal{G}$.

\begin{theorem}
$\abs{\mathcal{G}}=p^{\binom{p+1}{2}+1}$.
\end{theorem}

\begin{proof}
im$\varphi$ has $p^{\binom{p+1}{2}}$ elements. So, there are
\begin{displaymath}
\frac{p^{p^2+p+1}}{p^{\binom{p+1}{2}}}=p^{\binom{p+1}{2}+1}
\end{displaymath}
ghosts in $\PG(2,p)$.\qed
\end{proof}

\section{Conclusions}\label{sec:conclusions}
In this paper we have addressed the problem of classifying the subsets of $\PG(2,q)$ associated to a same power sum polynomial. By exploiting some similarities with the classical tomographic problem, we have shown that two sets with a same power sum polynomial differ, in the multiset sum sense, by a set whose power sum polynomial is the null one. In analogy to tomography, such a set has been called ghost. We have then counted the number of ghosts of $\PG(2,p)$ by arguing on the image of the function $\varphi$, which maps a (multi)set of $\PG(2,p)$ to the corresponding power sum polynomial.

The present paper may be considered as a starting point in the study of the connection between (multi)sets of points in $\PG(2,q)$ and their corresponding power sum polynomials within the framework of the polynomial technique for finite geometries. Therefore several future developments may be taken into account.

An immediate related research could be the study of the behavior of the power sum polynomials under the intersection of two sets.

Secondly, we aim to study the size of im$\varphi$ even when $h$ is greater than one. This would give the exact number of ghosts in every projective plane.

Moreover, the subgroup of ghosts deserves a deeper investigation. In particular, we would like to find, if possible, a set of generators for the ghosts.

Further, our investigation has dealt with the planar case, due to its connections with the tomographic problem. Another development is the extension to higher dimensions, to treat the problem in full generality.

\bibliographystyle{plain}
\bibliography{pspolynomials}

\end{document}